\documentclass[12pt]{amsart}

\usepackage{amssymb}
\usepackage{graphicx}
\usepackage{enumerate}
\usepackage{multirow}
\usepackage{amsmath,color}
\usepackage{hyperref}
\usepackage{url}

\newtheorem{thm}{Theorem}[section]
\newtheorem{cor}[thm]{Corollary}
\newtheorem{lem}[thm]{Lemma}
\newtheorem{prop}[thm]{Proposition}

\newtheorem{exam}[thm]{Example}
\numberwithin{equation}{section}
\theoremstyle{definition}
\newtheorem{rem}{Remark}[section]
\newtheorem{Observation}[thm]{Observation}

\newcommand{\seq}[1]{\langle #1\rangle}
\makeatletter
\@namedef{subjclassname@2020}{%
  \textup{2020} Mathematics Subject Classification}
\makeatother

\title{Neumaier Cayley graphs}

\author{Mojtaba Jazaeri}
\address{Department of Mathematics, Shahid Chamran University of Ahvaz, Ahvaz, Iran}
\email{M.Jazaeri@scu.ac.ir, M.Jazaeri@ipm.ir}
\begin{document}


\maketitle


\begin{abstract}
A Neumaier graph is a non-complete edge-regular graph with the property that it has a regular clique. In this paper, we study Neumaier Cayley graphs. We give a necessary and sufficient condition under which a Neumaier Cayley graph is a strongly regular Neumaier Cayley graph. We also characterize Neumaier Cayley graphs with small valency at most $10$.
\end{abstract}

\section{Introduction}

A regular graph with $n$ vertices and valency $k$ is an edge-regular graph with parameters $(n,k,\lambda)$ whenever every pair of adjacent vertices has a unique number $\lambda$ of common neighbors. Let $\Gamma$ be a graph with vertex set $V$. Then a subset $C$ of $V$ is a clique whenever each pair of vertices in $C$ are adjacent. A clique $C$ in the graph $\Gamma$ is called a regular clique whenever every vertex out of the clique $C$ has a constant number $a$ of its neighbors in $C$. The number $a$ is called the nexus of the regular clique $C$. Neumaier in \cite{Neumaier} studied regular cliques in an edge-regular graph and stated the question that is every edge-regular graph with a regular clique a strongly regular graph? Recently, a non-complete edge-regular graph with a regular clique is called a Neumaier graph and it has attracted a great deal of attention among authors. Soicher in \cite{Soicher} studied regular cliques in an edge-regular graph and Greaves and Koolen in \cite{GK} answered the Neumaier question by constructing of infinite examples of Neumaier graphs that are not strongly regular. After that, Abiad et all in \cite{ADDK} studied Neumaier graphs with few eigenvalues and proved that there is no Neumaier graph with exactly four distinct eigenvalues and Evans studied this family of graphs extensively in his Ph.D thesis \cite{Evans}. Moreover, Abiad et all have some other results on this topic in \cite{ACDKZ} recently. In this paper, we deal with Neumaier Cayley graphs. This paper has been organized as follows. We first give some facts on Neumaier graphs and then concentrate on Neumaier Cayley graphs. We give a necessary and sufficient condition under which a Neumaier Cayley graph is a strongly regular Neumaier Cayley graph. We also characterize Neumaier Cayley graphs with small valency at most $10$.

\section{Preliminaries}
In this paper, all graphs are undirected and simple, i.e., there are no loops or multiple edges. A regular graph with $n$ vertices and valency $k$ is an edge-regular graph with parameters $(n,k,\lambda)$ whenever every pair of adjacent vertices has a unique number $\lambda$ of common neighbors. Let $\Gamma$ be a graph with vertex set $V$. Then a subset $C$ of $V$ is a clique whenever each pair of vertices in $C$ are adjacent. A clique $C$ in the graph $\Gamma$ is called a regular clique whenever every vertex out of the clique $C$ has a constant number $a$ of its neighbors in $C$. The number $a$ is called the nexus of the regular clique $C$. A non-complete edge-regular graph with a regular clique is called a Neumaier graph. It is known that if $\Gamma$ is a Neumaier graph with a regular clique $C$ and nexus $a$, then all regular cliques of the graph $\Gamma$ has size $|C|$ and nexus $a$ (cf. \cite[Theorem~1]{Neumaier}). This means that the Neumaier graph $\Gamma$ has parameters $(n,k,\lambda;a,c)$, where $|C|=c$. It is easy to see that a Neumaier graph has diameter two or three because it has a regular clique with nexus $a \geq 1$. The following remarks and lemma are available in literature but in different structures, we state them for convenience.
\begin{rem} \label{Complete graph}
Let $\Gamma$ be a regular graph with a regular clique $C$. If the nexus of the regular clique $C$ equals to $|C|$, then this graph must be a complete graph.
\end{rem}
\begin{rem} \label{eigenvalues}
Let $\Gamma$ be a Neumaier graph with parameters $(n,k,\lambda;a,c)$. Then the regular clique $C$ gives rise to an equitable partition for the graph $\Gamma$ with the following quotient matrix.
\begin{center}
\begin{equation*}
  \left(
    \begin{array}{cc}
      c-1 & k-c+1 \\
      a & k-a \\
    \end{array}
  \right)
\end{equation*}
\end{center}
It follows that $k$ and $c-a-1$ are the eigenvalues of this graph.
\end{rem}
\begin{rem} \label{Complete multipartite graph}
Let $\Gamma$ be a Neumaier graph with parameters $(n,k,\lambda;a,2)$. Then $a=1$ and $\lambda=0$ and therefore this graph is the complete bipartite graph $K_{k,k}$. Similarly, if $\Gamma$ is a Neumaier graph with parameters $(n,k,0;a,c)$, then $c=2$ and $a=1$ and this graph is also the complete bipartite graph $K_{k,k}$. This implies that a Neumaier graph with diameter two and parameters $(n,k,\lambda;a,c)$ has at most $\max \{1+k+k(k-2),2k\}$ vertices.
\end{rem}
\begin{lem} \label{Counting}
Let $\Gamma$ be a $k$-regular graph with a regular clique of size $c$ and nexus $a$. Then $c(k-c+1)=(n-c)a$.
\end{lem}
\begin{proof}
By the double counting method, the number of edges between the regular clique $C$ and out of $C$ equal to $c(k-c+1)$ and $(n-c)a$ which completes the proof.

A strongly regular graph with parameters $(n,k,\lambda,\mu)$ is a $k$-regular graph with $n$ vertices such that two vertices has $\lambda$ and $\mu$ common neighbors depending on whether these two vertices are adjacent or non-adjacent, respectively.

Let $G$ be a group and $S$ be an inverse-closed subset without the identity element of the group $G$. Then the Cayley graph $Cay(G,S)$ is a graph whose vertex set is $G$ and two vertices $a$ and $b$ are adjacent (which is denoted by $a \sim b$) whenever $ab^{-1} \in S$. We call the subset $S$ of the group $G$, the connection set of the Cayley graph $Cay(G,S)$. It is known that a Cayley graph $Cay(G,S)$ is connected if and only if the subgroup generated by $S$ (which is denoted by $\seq{S}$) equals $G$. If $C$ is a regular clique in the Cayley graph $Cay(G,S)$, then it is easy to see that $Cg^{-1}$ is also a regular clique for every $g \in G$ and therefore we can suppose that there exist a regular clique containing the identity element of the group $G$.
\end{proof}
\section{Neumaier Cayley graphs} \label{Section diameter two}
Recall that a Neumaier graph has diameter two or three. In this section, we study Neumaier Cayley graphs that they have diameter two (see Observation \ref{Diameter 2} below). Some of the results in this section are known in literature in different formats but most of them are new. We state the results that are necessary to use in other sections for convenience.

Let $Cay(G,S)$ be a Neumaier Cayley graph over the group $G$ with parameters $(n,k,\lambda;a,c)$ and regular clique $C$ containing the identity element $e$. Then $C \setminus \{e\} \subset S$ and every vertex in $C \setminus \{e\}$ has $c-2$ neighbors in $C \setminus \{e\}$, $\lambda-c+2$ neighbors in $S \setminus C$ and therefore $k-\lambda -1$ neighbors in $G \setminus (S\cup\{e\})$. Furthermore, every vertex in $S \setminus C$  has $a-1$ neighbors in $C \setminus \{e\}$, $\lambda-a+1$ neighbors in $S \setminus C$, and therefore $k-\lambda -1$ neighbors in $G \setminus (S\cup\{e\})$ (see Figure \ref{Fig1}). This implies that the number of edges from the vertices in $S \setminus C$ to the vertices in $C \setminus \{e\}$ is
\begin{equation}\label{equation 1}
 (k-c+1)(a-1)=(c-1)(\lambda - c +2)
\end{equation}
\begin{center}
\begin{figure}[h]
 \centering
 \centerline{\includegraphics*[height=7cm]{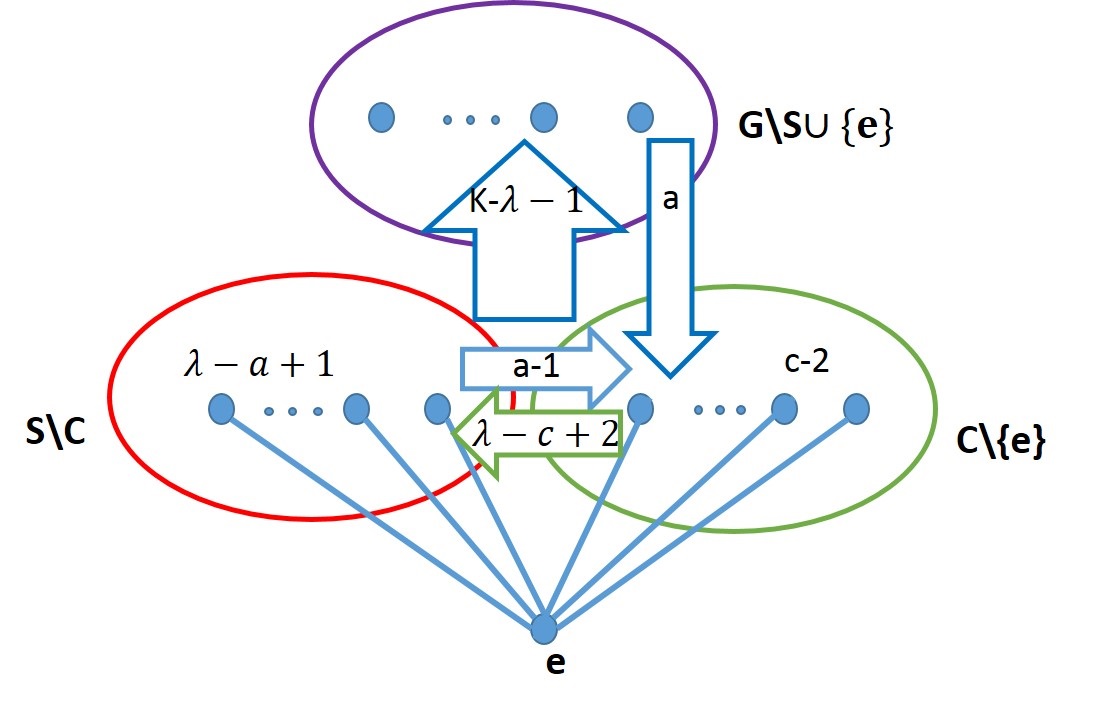}}
 \label{diagram}
 \caption{A Neumaier Cayley graph with diameter two and parameters $(n,k,\lambda;a,c)$} \label{Fig1}
\end{figure}
\end{center}
Moreover, the number of edges from the vertices in $C \setminus \{e\}$ to the vertices in $G \setminus (S\cup\{e\})$ is $(c-1)(k-\lambda -1)$ and the number of edges from the vertices in $G \setminus (S\cup\{e\})$ to the vertices in $C \setminus \{e\}$ is $(n-k-1)a$ since the regular clique $C$ has nexus $a$. It follows that
\begin{equation}\label{equation 2}
 (c-1)(k-\lambda -1)=(n-k-1)a
\end{equation}
Moreover, if every vertex in $G \setminus (S\cup\{e\})$ has a unique number of neighbors $b$ in $S \setminus \{C\}$, then this Neumaier Cayley graph is a strongly regular Neumaier graph with parameters $(n,k,\lambda,\mu;a,c)$, where $\mu=a+b$. To see this let $x$ and $y$ be two other vertices at distance two. Then $xx^{-1}=e$ and $yx^{-1}$ are at distance two. This implies that $x$ and $y$ has $\mu$ common neighbors because $e$ and $yx^{-1}$ has $\mu$ common neighbors. With the above notation for a Neumaier Cayley graph with diameter two, we can conclude the following characterization for a Neumaier Cayley graph to be a Strongly regular Neumaier Cayley graph.
\begin{thm} \label{Strongly regular Neumaier graph}
Let $\Gamma$ be a Neumaier Cayley graph with parameters $(n,k,\lambda;a,c)$. Then $\Gamma$ is a strongly regular Neumaier Cayley graph if and only if every vertex in $G \setminus (S\cup\{e\})$ has a unique number of neighbors in $S \setminus C$.
\end{thm}
Therefore we can conclude the following corollary because the number of edges between $G \setminus (S\cup\{e\})$ and $S \setminus C$ is constant.
\begin{cor} \label{Strongly regular Neumaier graph counting}
Let $\Gamma$ be a strongly regular Neumaier Cayley graph with parameters $(n,k,\lambda,\mu;a,c)$. Then $(k-c+1)(k-\lambda-1)=(n-k-1)(\mu-a)$.
\end{cor}
\begin{rem} \label{nexus one}
Let $\Gamma$ be a Neumaier Cayley graph with parameters $(n,k,\lambda;1,c)$. Then $\lambda=c-2$ by Equation \ref{equation 1} and the regular induced subgraph on $S \setminus C$ has valency $c-2$. This implies that $k-c+1>c-2$ and therefore $k-2c+3>0$.
\end{rem}
\begin{thm} \label{The number of edges}
Let $\Gamma$ be a Neumaier Cayley graph with parameters $(n,k,\lambda;a,c)$. Then this graph has at least the following number of edges.
\begin{equation*}
k(k-\lambda)+(k-c+1)(a-1)+\frac{(k-c+1)(\lambda-a+1)}{2}+\frac{(c-1)(c-2)}{2}
\end{equation*}
Furthermore, if the Neumaier Cayley graph $\Gamma$ admits this bound, then it is a strongly regular Neumaier Cayley graph with parameters $(n,k,\lambda,k)$.
\end{thm}
\begin{proof}
With the above notation, the number of edges from the vertex $e$ into the set $S$ is $k$. Moreover, the number of edges from the vertices in the set $S$ into the set $G \setminus (S\cup\{e\})$ equals $k(k-\lambda -1)$. Furthermore, the number of edges from the vertices in the set $S$ into itself equals
\begin{equation*}
(k-c+1)(a-1)+\frac{(k-c+1)(\lambda-a+1)}{2}+\frac{(c-1)(c-2)}{2}
\end{equation*}
This implies that the number of edges is at least
\begin{equation*}
k(k-\lambda)+(k-c+1)(a-1)+\frac{(k-c+1)(\lambda-a+1)}{2}+\frac{(c-1)(c-2)}{2}.
\end{equation*}
Moreover, if the number of edges is exactly this number, then there is no edge in $G \setminus (S\cup\{e\})$ and therefore every vertex in $G \setminus (S\cup\{e\})$ is adjacent to $k$ vertices in the set $S$. This implies that this graph is indeed a strongly regular Neumaier Cayley graph with parameters $(n,k,\lambda,k;a,c)$ and this completes the proof.
\end{proof}
\begin{rem} \label{edge-regular with parameter}
Let $\Gamma$ be an edge-regular graph with parameter $(n,k,1)$. Then it is easy to see that this graph must have even degree since $\lambda=1$. Let $\Gamma$ be a Neumaier Cayley graph with diameter two and parameters $(n,k,\lambda;1,3)$. Then the parameter $\lambda=1$ by Remark \ref{nexus one} and this graph must have even degree.
\end{rem}
\begin{lem} \label{Special cases}
Let $\Gamma$ be a Neumaier Cayley graph with parameters $(n,k,\lambda;a,c)$. If $c \geq k$, then the graph $\Gamma$ is the cycle graph $C_{4}$.
\end{lem}
\begin{proof}
If $ c > k$, then $c=k+1$ and $\lambda=k-1$. This implies that the graph $\Gamma$ is the complete graph on $n=k+1$ vertices by Equation \ref{equation 2} which is impossible. If $ c = k \geq 3$, then the parameter $\lambda$ equals $k-2$ ($\lambda=k-1$ is impossible by the previous argument). This implies that $k=(n-k)a$ by Lemma \ref{Counting} and $k-1=(n-k-1)a$ by Equation \ref{equation 2}. It follows that $a=1$ and $n=2k$. On the other hand, $k<3$ by Remark \ref{nexus one} which implies that $k=2$. Therefore the graph $\Gamma$ is the cycle graph $C_{4}$ and this completes the proof.
\end{proof}

Let $-m$ be the smallest eigenvalue of a strongly regular Neumaier graph with parameters $(n,k,\lambda,\mu;a,c)$. Then the regular clique $C$ admit the Hoffman bound. Namely, $c=1+\frac{k}{m}$ (cf. \cite[Proposition~1.3.2]{BCN}). Moreover, the nexus of the regular clique $C$ is $a=\frac{\mu}{m}$. This implies that $k$, $c-a-1$, and $-m=-\frac{\mu}{a}$ are the all three distinct eigenvalues of the strongly regular Neumaier graph by Remark \ref{eigenvalues}. Therefore we can conclude the following result.
\begin{prop} \label{Integer eigenvalues}
Let $\Gamma$ be a strongly regular Neumaier graph with parameters $(n,k,\lambda,\mu;a,c)$. Then it has integer eigenvalues $k$, $c-a-1$, and $-\frac{\mu}{a}$.
\end{prop}
\section{Neumaier Cayley graphs and Schur rings}
There exist several examples of Neumaier Cayley graphs in literature, for example the construction by Greaves and Koolen in \cite{GK}. We note that a Neumaier Cayley graph has diameter two (see Observation \ref{Diameter 2} below). In this section, we use Schur ring to study Neumaier Cayley graphs. We also use the results of \S \ref{Section diameter two} to characterize Neumaier Cayley graphs with small valency at most $10$. We also distinguish between strongly regular Neumaier Cayley graph and strictly Neumaier Cayley graph and give a necessary and sufficient condition under which a Neumaier Cayley graph is a strongly regular Neumaier Cayley graph.

Let $G$ be a group and $R$ a commutative ring with identity. Then the group algebra $RG$ consists of the elements of form $\sum_{g \in G}a_{g}g$, where $a_{g} \in R$, with the following operations.
\begin{equation*}
\sum_{g \in G}a_{g}g+\sum_{g \in G}b_{g}g=\sum_{g \in G}(a_{g}+b_{g})g,
\end{equation*}
\begin{equation*}
(\sum_{g \in G}a_{g}g)(\sum_{h \in G}b_{h}h)=\sum_{g,h \in G}(a_{g}b_{h})gh,
\end{equation*}
and the scaler multiplication
\begin{equation*}
c\sum_{g \in G}a_{g}g=\sum_{g \in G}(ca_{g})g.
\end{equation*}
Let $T$ be a subset of the group $G$. Then the element $\sum_{t \in T}t$, in this algebra, is denoted by $\overline{T}$. Moreover, for every element $a=\sum_{g \in G}a_{g}g$, the element $\sum_{g \in G}a_{g}g^{-1}$ is denoted by $a^{(-1)}$. Let $\{T_{0},T_{1},\ldots,T_{d}\}$, where $T_{0}=\{e\}$, be the partition of the group $G$. Then the subalgebra $\mathcal{S}$ generated by $\alpha=\{\overline{T_{0}},\overline{T_{1}},\ldots,\overline{T_{d}}\}$ such that the $R$-module $\mathcal{S}$ is generated by $\alpha$ with the property that $a^{(-1)} \in \mathcal{S}$ for every element $a \in \mathcal{S}$ is called Schur ring over the ring $R$, and in this case $\alpha$ is called the simple basis of the Schur ring $\mathcal{S}$.

Let $Cay(G,S)$ be a Cayley graph with a regular clique $C$ and nexus $a$. Then we can suppose that the identity element $e \in C$ and $ C \setminus \{e\} \subseteq S$ because if $x \in C$, then $Cx^{-1}$ is also a regular clique with nexus $a$. In the sequel, we can suppose that the regular clique $C$ contains the identity element $e$. We observe the following identity in the group algebra $\mathbb{Z}G$.
\begin{equation} \label{Equation 1 Cayley}
\overline{S} \, \overline{C}=a\overline{G \setminus C}+(|C|-1)\overline{C}.
\end{equation}
Therefore we can conclude the following observation and proposition.
\begin{Observation} \label{Diameter 2}
A non-complete connected Cayley graph with a regular clique has diameter two.
\end{Observation}
\begin{prop} \label{Cayley graphs with a regular clique}
Let $Cay(G,S)$ be a non-complete connected Cayley graph over the group $G$ with the connection set $S$. Then the Cayley graph $Cay(G,S)$ has a regular clique $C$ with nexus $a$ if and only if the Cayley graph $Cay(G,S)$ has diameter two and
\begin{equation*}
\overline{S} \, \overline{C}=a\overline{G \setminus C}+(|C|-1)\overline{C},
\end{equation*}
where $a$ is a natural number.
\end{prop}
\begin{lem} \label{regular clique of Size 3 Cayley}
Let $Cay(G,S)$ be a connected Cayley graph over the abelian group $G$. Let $C$ be the regular clique of size $3$ and nexus one containing the identity element $e$ of the group $G$. Then $C$ is a subgroup of order $3$ in the abelian group $G$.
\end{lem}
\begin{proof}
Let $C$ be a regular clique $C$ of size $3$ and nexus one and $C=\{e,a,a^{-1}\}$. Then there exist two cases. Firstly, $a^{2} \in S \setminus C$ and $a^{2}$ is adjacent to $e$ and $a$, a contradiction with the nexus of the regular clique $C$. Secondly, $a^{2}=a^{-1}$ and $C=\{e,a,a^{-1}\}$ is a subgroup of order $3$ as desired. If $C=\{e,a,b\}$, where $a$ and $b$ are two involutions, then $ab=ba$ is adjacent to both $a$ and $b$, a similar contradiction with the nexus of the regular clique $C$ and this completes the proof.
\end{proof}
\begin{rem} \label{Subgroup regular code}
Let $Cay(G,S)$ be a Connected Cayley graph with a regular clique $C$ and nexus $a$. If $C$ is a subgroup of the group $G$, then $|S|=([G:C]-1)a+|C|-1$ because each of the right cosets of $C$ in $G$ is also a regular clique with nexus $a$ and $C \setminus \{e\} \subset S$. We also conclude that $\overline{S \setminus C} \, \overline{C}=a\overline{G \setminus C}$ by Equation \ref{Equation 1 Cayley}.
\end{rem}
\subsection{Strongly regular Neumaier Cayley graphs}
Let $Cay(G,S)$ be a strongly regular Cayley graph with parameters $(|G|,|S|,\lambda,\mu)$. Then the connection set $S$ is a partial difference set in the group $G$ and the following identity holds (cf. \cite[Proposition~1.1 and Theorem~1.3]{MA}).
\begin{equation*}
\overline{S}^{2}=\mu\overline{G}+(\lambda-\mu)\overline{S}+(|S|-\mu)e.
\end{equation*}
If the strongly regular Cayley graph $Cay(G,S)$ has a regular clique $C$ containing the identity element $e$ with nexus $a$, then by Equation \ref{Equation 1 Cayley} we can conclude the following proposition.
\begin{prop}
Let $Cay(G,S)$ be a Cayley graph over the group $G$ with the connection set $S$. Then $Cay(G,S)$ is a strongly regular Neumaier Cayley graph with parameters $(|G|,|S|,\lambda,\mu;a,c)$ and the regular clique $C$ of size $c$ and nexus $a$ containing the identity element of the group $G$ if and only if
\begin{equation*}
\overline{S}^{2}=\mu\overline{G}+(\lambda-\mu)\overline{S}+(|S|-\mu)e,
\end{equation*}
\begin{equation*}
\overline{S} \, \overline{C}=a\overline{G \setminus C}+(c-1)\overline{C},
\end{equation*}
for some non-negative integers $\lambda$, $\mu$, $a$, and $c$.
\end{prop}
This implies that
\begin{equation*}
\overline{S}^{2} - \overline{S} \, \overline{C}=\overline{S}(\overline{S} - \overline{C})=\overline{S} \, \overline{S \setminus C}=
\end{equation*}
\begin{equation} \label{Equation 2 Cayley}
(\mu-a)\overline{G \setminus (S\cup\{e\})}+(\lambda-a)\overline{S \setminus C}+(\lambda-|C|+1)\overline{C \setminus \{e\}}+(|S|-|C|+1)e.
\end{equation}
Therefore by Theorem \ref{Strongly regular Neumaier graph}, we can conclude the following characterization for a Neumaier Cayley graph to be a strongly regular Neumaier Cayley graph.
\begin{thm}
Let $Cay(G,S)$ be a Neumaier Cayley graph with parameters $(|G|,|S|,\lambda;a,c)$ and the regular clique $C$ of size $c$ and nexus $a$ containing the identity element of the group $G$. Then the Cayley graph $Cay(G,S)$ is a strongly regular Neumaier Cayley graph with parameters $(|G|,|S|,\lambda,\mu;a,c)$ if and only if
\begin{equation*}
\overline{S} \, \overline{S \setminus C}=
\end{equation*}
\begin{equation*}
(\mu-a)\overline{G \setminus (S\cup\{e\})}+(\lambda-a)\overline{S \setminus C}+(\lambda-|C|+1)\overline{C \setminus \{e\}}+(|S|-|C|+1)e,
\end{equation*}
for a non-negative integer $\mu$.
\end{thm}
\begin{exam} \label{Example}
The complete multipartite graph $K_{\underbrace{m,m,\ldots,m}_{n-times}}$ is a strongly regular Neumaier graph with parameters $(nm,(n-1)m,(n-2)m,(n-1)m;n,n-1)$. It is also a Cayley graph over the group $G=\mathbb{Z}_{m}^{n}$ with the connection set $S=G \setminus H$, where $H$ is a subgroup of the group $G$ of order $m$ (cf. \cite[Prop.~2.6]{AJ}). The lattice graph $L_{2}(n,n)$ is the line graph of the complete bipartite graph $K_{n,n}$. It is a strongly regular Neumaier graph with parameters $(n^{2},2(n-1),n-2,2;n,1)$. It is also a Cayley graph over the group $\mathbb{Z}_{n} \times \mathbb{Z}_{n}$ with the connection set $\{(0,1),(0,2),\ldots,(0,n-1),(1,0),(2,0),\ldots,(n-1,0)\}$ (cf. \cite[Thm.~4.7]{ADJ}).
\end{exam}
Let $G$ be an abelian group and $C$ be a subgroup of the group $G$ that is a regular clique with nexus $a$. Then by Equation \ref{Equation 2 Cayley} we can conclude that
\begin{equation*}
\overline{S \setminus C} \, \overline{S \setminus C}=
\end{equation*}
\begin{equation*}
(\mu-2a)\overline{G \setminus (S\cup\{e\})}+(\lambda-2a)\overline{S \setminus C}+(\lambda-|C|+1)\overline{C \setminus \{e\}}+(|S|-|C|+1)e
\end{equation*}
since $\overline{S \setminus C} \, \overline{C}=a\overline{G \setminus C}$.
On the other hand, the complement of a strongly regular graph with parameters $(n,k,\lambda,\mu)$ is also a strongly regular graph with parameters $(n,n-k-1,n-2k-2+\mu,n-2k+\lambda)$. Therefore the following identity holds (cf. \cite[Proposition~1.1 and Theorem~1.3]{MA}).
\begin{equation*}
\overline{G \setminus (S\cup\{e\})}^{2}=(|G|-2|S|+\lambda)\overline{G}+(\mu- \lambda -2)\overline{G \setminus (S\cup\{e\})}+(|S|-1-\lambda)e.
\end{equation*}
On the other hand, the regular clique $C$ in the strongly regular Cayley graph $Cay(G,S)$ can be considered as an independent set in its complement in such a way that every vertex out of this independent set has $|C|-a$ adjacent vertex in $C$. It follows that
\begin{equation*}
\overline{G \setminus (S\cup\{e\})} \, \overline{C}=(|C|-a)\overline{G \setminus C}.
\end{equation*}
On the other hand,
\begin{equation*}
(|G|-|S|-1)\overline{G}=\overline{G \setminus (S\cup\{e\})} \, \overline{G}=\overline{G \setminus (S\cup\{e\})} \, (\overline{C}+\overline{S \setminus C}+\overline{G \setminus (S\cup\{e\})}).
\end{equation*}
This implies that
\begin{equation*}
\overline{G \setminus (S\cup\{e\})} \, \overline{S \setminus C}=
\end{equation*}
\begin{equation*}
(|S|-\mu-|C|+a+1)\overline{G \setminus (S\cup\{e\})}+(|S|-|C|-\lambda+a-1)\overline{S \setminus C}+(|S|-\lambda-1)\overline{C \setminus \{e\}}.
\end{equation*}
Therefore we can conclude the following proposition.
\begin{prop}
Let $Cay(G,S)$ be a strongly regular Neumaier Cayley graph over the abelian group $G$ with parameters $(|G|,|S|,\lambda,\mu;a,c)$. Let $C$ be a subgroup of the group $G$ that is a regular clique with nexus $a$. Then the subalgebra generated by $\{\{e\},C \setminus \{e\},S \setminus C, G \setminus S\}$ is a Schur ring in the group ring $\mathbb{Z}G$.
\end{prop}

\subsection{Strongly regular Neumaier circulant graphs}
It is well known that a circulant graph on $n$ vertices is a Cayley graph over the cyclic group $\mathbb{Z}_{n}$. Furthermore, the Paley graphs of prime order $p$, where $p \equiv 1$ (mod $4$), are the only non-trivial strongly regular Cayley graphs (cf. \cite[Corollary~$5.7$]{MA}). On the other hand, the Paley graphs of prime order $p$, where $p \equiv 1$ (mod $4$), have non-integer eigenvalues. This implies that there is no non-trivial strongly regular Neumaier circulant graph by Proposition \ref{Integer eigenvalues}.
\subsection{Neumaier Cayley graphs with small valency}
Recall that a Neumaier Cayley graph has diameter two (see Observation \ref{Diameter 2}). In this section, we deal with Neumaier Cayley graphs with small valency at most $10$ and therefore, by the results of previous sections, we characterize Neumaier Cayley graphs with small valency at most $10$. We note that the feasible parameters for a strictly Neumaier graph up to $64$ vertices are available in \cite[Table~1]{ACDKZ} and a Neumaier graph with small valency at most $8$ has less than $64$ vertices. In this section, we somewhere discover again these feasible parameters with a different approach and characterize Cayley graphs among them.

It is easy to see that a Neumaier graph with valency $2$ has a regular clique of size $2$ with nexus one. This implies that it is indeed the cycle graph $C_{4}$ which is a strongly regular Neumaier Cayley graph.

Let $\Gamma$ be a Neumaier Cayley graph with diameter two and valency $3$. Then it has at most $7$ vertices by Remark \ref{Complete multipartite graph}. Moreover, the graph $\Gamma$ has even number of vertices because it has odd valency. It follows that $|G|=6$ and therefore $c=2$ by Remark \ref{nexus one}. This implies that the graph $\Gamma$ is indeed the complete bipartite graph $K_{3,3}$ by Remark \ref{Complete multipartite graph} which is a strongly regular Neumaier Cayley graph.

Let $\Gamma$ be a Neumaier Cayley graph with diameter two and parameters $(n,4,\lambda;a,c)$. Then it has at most $13$ vertices by Remark \ref{Complete multipartite graph}. There are three possible cases by Lemma \ref{Counting}; firstly, $c=2$ and therefore it has parameters $(8,4,0;1,2)$ by Remark \ref{Complete multipartite graph}. This implies that the graph $\Gamma$ is the complete bipartite graph $K_{4,4}$. Secondly, $c=3$ and therefore it has parameters $(6,4,2;2,3)$ by Equation \ref{equation 1}. This graph is indeed the complete multipartite graph $K_{2,2,2}$ which is a strongly regular Neumaier Cayley graph (see Example \ref{Example}). Thirdly, $c=3$ and therefore it has parameters $(9,4,1;1,3)$ by Equation \ref{equation 1} but there is no strictly Neumaier graph with such parameters (cf. \cite[Table~1]{ACDKZ}). Furthermore, there is a unique strongly regular graph with $9$ vertices and valency $4$ which is the lattice graph $L_{2}(3)$ and it is a strongly regular Neumaier Cayley graph (see Example \ref{Example}).

Let $\Gamma$ be a Neumaier Cayley graph with diameter two and parameters $(n,5,\lambda;a,c)$. Then it has at most $21$ vertices by Remark \ref{Complete multipartite graph}. If $c=2$, then $n=10$ and $a=1$ by Lemma \ref{Counting}. This graph is indeed the complete bipartite graph $K_{5,5}$ by Remark \ref{Complete multipartite graph} which is a strongly regular Neumaier Cayley graph. If $c=3$, then $n=12$ and $a=1$ by Lemma \ref{Counting}. This implies that the parameter $\lambda=1$ by Equation \ref{equation 1} which is impossible by Remark \ref{edge-regular with parameter}. If $c=4$, then there are two possible cases by Lemma \ref{Counting}; firstly $n=8$ and $a=2$ which is impossible by Equation \ref{equation 2}, and secondly $n=10$ and $a=1$ which is impossible by Remark \ref{nexus one}.

Let $\Gamma$ be a Neumaier Cayley graph with diameter two and parameters $(n,6,\lambda;a,c)$. Then it has at most $31$ vertices by Remark \ref{Complete multipartite graph}. If $c=2$, then $n=12$ and $a=1$ by Lemma \ref{Counting}. This graph is indeed the complete bipartite graph $K_{6,6}$ by Remark \ref{Complete multipartite graph} which is a strongly regular Neumaier Cayley graph. If $c=3$, then there are two possible cases  by Lemma \ref{Counting}; firstly $n=15$ and $a=1$, secondly $n=9$ and $a=2$ that it is easy to see that it is the complete multipartite graph $K_{3,3,3}$ which is a strongly regular Neumaier Cayley graph (see Example \ref{Example}). If $n=15$, then the graph $\Gamma$ has parameters $(15,6,1;3,1)$ by Equation \ref{equation 2} but there is no strictly Neumaier graph with such parameters (cf. \cite[Table~1]{ACDKZ}). Furthermore, there is a unique strongly regular graph with $15$ vertices and valency $6$ that it is the collinearity graph of the unique generalized quadrangle $GQ(2,2)$ (cf. \cite[\S~10.5]{BM}). We check it with \verb"GAP" \cite{GAP} and it is a strongly regular Neumaier graph but is not a Cayley graph.
If $c=4$, then there are three possible cases by Lemma \ref{Counting}; firstly $n=16$ and $a=1$, secondly $n=10$ and $a=2$, and thirdly $n=8$ and $a=3$. If $n=16$, then it has parameters $(16,6,2;1,4)$ by Equation \ref{equation 2}. There is no strictly Neumaier graph with such parameters (cf. \cite[Table~1]{ACDKZ}). Furthermore, there are two strongly regular graphs with $16$ vertices and valency $6$ (cf. \cite[Chap.~12]{BM}); the Hamming graph $H(2,4)$ and the Shrikhande graph. We checked them with \verb"GAP" and the Hamming graph $H(2,4)$ is a strongly regular Neumaier graph but the Shrikhande graph is not a Neumaier graph. We also note that the Hamming graph is also a Cayley graph (cf. \cite[\S 3.3]{DJ}). If $n=10$, then it has parameters $(10,6,3;2,4)$ by Equation \ref{equation 2}. It is easy to see that this graph is indeed the complement of the Petersen graph. This graph is a strongly regular Neumaier graph but it is not a Cayley graph because it is well known the Petersen graph is not a Cayley graph. If $n=8$, then it has parameters $(8,6,4;3,4)$ by Equation \ref{equation 2}. This graph is indeed the complete multipartite graph $K_{2,2,2,2}$ that is a strongly regular Neumaier Cayley graph (see Example \ref{Example}). If $c=5$, then there are two possible cases by Lemma \ref{Counting}; firstly $n=15$ and $a=1$ which is impossible by Remark \ref{nexus one}, and secondly $n=10$ and $a=2$ which is impossible by Equation \ref{equation 2}.

Let $\Gamma$ be a Neumaier Cayley graph with diameter two and parameters $(n,7,\lambda;a,c)$. Then it has at most $43$ vertices by Remark \ref{Complete multipartite graph}. If $c=2$, then $n=14$ and $a=1$ by Lemma \ref{Counting}. This graph is indeed the complete bipartite graph $K_{7,7}$ by Remark \ref{Complete multipartite graph} which is a strongly regular Neumaier Cayley graph. If $c=3$, then the there is only one possible case by Lemma \ref{Counting} that is $n=18$ and $a=1$ which is impossible by Remark \ref{edge-regular with parameter}. If $c=4$, then there are two possible cases by Lemma \ref{Counting}; firstly $n=20$ and $a=1$, secondly $n=12$ and $a=2$ which is impossible by Equation \ref{equation 2}. If $n=20$, then it has parameters $(20,7,2;1,4)$ by Equation \ref{equation 2}. There is neither strictly Neumaier graph with such parameters (cf. \cite[Table~1]{ACDKZ}) nor strongly regular graph (cf. \cite[Chap.~12]{BM}). If $c=5$, then there are two possible cases by Lemma \ref{Counting}; firstly $n=20$ and $a=1$ which is impossible by Remark \ref{nexus one}, and secondly $n=10$ and $a=3$ which is impossible by Equation \ref{equation 2}. Finally, if $c=6$, then there are two possible cases by Lemma \ref{Counting}; firstly $n=18$ and $a=1$ which is impossible by Remark \ref{nexus one}, and secondly $n=12$ and $a=2$ which is impossible by Equation \ref{equation 2}.

Let $\Gamma$ be a Neumaier Cayley graph with diameter two and parameters $(n,8,\lambda;a,c)$. Then it has at most $57$ vertices by Remark \ref{Complete multipartite graph}. If $c=2$, then $n=16$ and $a=1$ by Lemma \ref{Counting}. This graph is indeed the complete bipartite graph $K_{8,8}$ by Remark \ref{Complete multipartite graph} which is a strongly regular Neumaier Cayley graph. If $c=3$, then there are two possible cases by Lemma \ref{Counting}; firstly $n=18$ and $a=1$ which is impossible by Equation \ref{equation 2}, and secondly $n=12$ and $a=2$. If $n=12$, then it has parameters $(12,8,4;2,3)$ that is indeed the complete multipartite graph $K_{4,4,4}$ which is a strongly regular Neumaier Cayley graph (see Example \ref{Example}).
If $c=4$, then there are two possible cases by Lemma \ref{Counting}; firstly $n=24$ and $a=1$, secondly $n=14$ and $a=2$ which is impossible by Equation \ref{equation 2}. If $n=24$, then it has parameters $(24,8,2;1,4)$ by Equation \ref{equation 1}. Furthermore, there is no strongly regular graph with $24$ vertices and valency $8$ (cf. \cite[Chap.~12]{BM}). Moreover, there are only four vertex-transitive strictly Neumaier graphs on $24$ vertices with parameters $(24,8,2;1,4)$ (cf. \cite[\S~4.4.1]{Evans}). These four vertex-transitive strictly Neumaier graphs are available in \verb"GRAPE" format of \verb"GAP" in \cite[\S~4.A]{Evans}. We checked them with \verb"GAP" \cite{GAP} and all of them are Cayley graphs. The first and the third ones are over $G=S_{4}$, the symmetric group on four letters, with the following connection sets, respectively.
\begin{equation*}
S_{1}=\{ (1,3)(2,4), (1,4)(2,3), (1,3,4), (1,4,2), (1,4,3), (1,2,4), (1,2,4,3), (1,3,4,2)\}),
\end{equation*}
\begin{equation*}
S_{3}=\{(1,4)(2,3), (1,3)(2,4), (1,2,4), (1,4,3), (1,4,2), (1,3,4), (1,4), (2,3)\}).
\end{equation*}
The second and the last ones are over $G=\mathbb{Z}_{2} \times A_{4}=$
\begin{equation*}
\seq{a,b,c,d| a^{2}=b^{3}=c^{2}=d^{2}=a^{-1}b^{-1}ab=a^{-1}c^{-1}ac=a^{-1}d^{-1}ad=
\end{equation*}
\begin{equation*}
d^{-1}c^{-1}dc=c^{-1}b^{-1}cbd^{-1}c^{-1}=d^{-1}b^{-1}dbc^{-1}=e}
\end{equation*}
with the following connection sets, respectively.
\begin{equation*}
S_{2}=\{cd,ad,d,bcd,bc,b^{2}d,b^{2}cd,ac\},
\end{equation*}
\begin{equation*}
S_{4}=\{a,b,c,d,bcd,b^{2},ac,b^{2}d\}.
\end{equation*}
If $c=5$, then there are two possible cases by Lemma \ref{Counting}; firstly $n=25$ and $a=1$, secondly $n=15$ and $a=2$ that these are impossible by Equation \ref{equation 2}.
Furthermore, if $c=6$, then there are two possible cases by Lemma \ref{Counting}; firstly $n=25$ and $a=1$, secondly $n=15$ and $a=2$ that these are impossible by Equation \ref{equation 2}. Finally, if $c=7$, then there are two possible cases by Lemma \ref{Counting}; firstly $n=21$ and $a=1$ which is impossible by Remark \ref{nexus one}, secondly $|G|=14$ and $a=2$ which is impossible by Equation \ref{equation 2}.

Let $\Gamma$ be a Neumaier Cayley graph with diameter two and parameters $(n,9,\lambda;a,c)$. Then it has at most $73$ vertices by Remark \ref{Complete multipartite graph}. If $c=2$, then $n=18$ and $a=1$ by Lemma \ref{Counting}. This graph is indeed the complete bipartite graph $K_{9,9}$ by Remark \ref{Complete multipartite graph} which is a strongly regular Neumaier Cayley graph. If $c=3$, then there is only one possible case by Lemma \ref{Counting} that is $n=24$ and $a=1$ which is impossible by Remark \ref{edge-regular with parameter}. If $c=4$, then there are three possible cases by Lemma \ref{Counting}; firstly $n=28$ and $a=1$, secondly $n=16$ and $a=2$, and thirdly $n=12$ and $a=3$. If $n=28$, then there are two known non-isomorphic vertex-transitive Neumaier graphs with parameters $(28,9,2;1,4)$ (cf. \cite[\S~4.A]{Evans}). It turns out that both of them are Cayley graphs as follows. The first one is over the cyclic group $\mathbb{Z}_{28}=\seq{a}$ with the connection set $S=\{a,a^{-1},a^{4},a^{-4},a^{5},a^{-5},a^{7},a^{-7},a^{14}\}$, and the second one is over the abelian group $\mathbb{Z}_{2} \times \mathbb{Z}_{14}=\seq{a,b|a^{2}=b^{14}=1,ab=ba}$ with the connection set $S=\{a,b,b^{-1},b^{7},ab^{2},ab^{-2},ab^{3},ab^{-3},ab^{7}\}$. If $n=16$, then there is only one strictly Neumaier graph with parameters $(16,9,4;2,4)$ up to isomorphism (cf. \cite[\S~4.A]{Evans}) and this graph is a Cayley graph over the dihedral group $D_{16}=\seq{a,b|a^{8}=b^{2}=(ba)^{2}=e}$ with the connection set $\{a,a^{-1},a^{2},a^{-2},b,ba,ba^{3},ba^{4},ba^{6}\}$. This graph is indeed the smallest strictly Neumaier graph and can also be constructed as Cayley graph over the abelian group $\mathbb{Z}_{8} \times \mathbb{Z}_{2}$ (cf. \cite[\S~6]{EGP}). If $n=12$, then this graph is indeed the complete multipartite graph $K_{3,3,3,3}$ which is a strongly regular Neumaier Cayley graph (see Example \ref{Example}). If $c=5$, then there is only one possible case by Lemma \ref{Counting} that is $n=30$ and $a=1$ and therefore it has parameters $(30,9,3;1,5)$ by Equation \ref{equation 1}. There is neither strictly Neumaier graph with such parameters (cf. \cite[Table~1]{ACDKZ}) nor strongly regular graph (cf. \cite[Chap.~12]{BM}). If $c=6$, then there are three possible cases by Lemma \ref{Counting}; firstly $n=30$ and $a=1$ which is impossible by Remark \ref{nexus one}, secondly $n=18$ and $a=2$ which is impossible by Equation \ref{equation 2}, thirdly $n=12$ and $a=4$ which is impossible by Equation \ref{equation 2}. If $c=7$, then there are two possible cases by Lemma \ref{Counting}; firstly $n=28$ and $a=1$ which is impossible by Remark \ref{nexus one}, secondly $n=14$ and $a=3$ which is impossible by Equation \ref{equation 2}. Finally, if $c=8$, then there are two possible cases by Lemma \ref{Counting}; firstly $n=25$ and $a=1$, secondly $n=18$ and $a=2$ which are impossible by Equation \ref{equation 2}.

Let $\Gamma$ be a Neumaier Cayley graph with diameter two and parameters $(n,10,\lambda;a,c)$. Then it has at most $91$ vertices by Remark \ref{Complete multipartite graph}. If $c=2$, then $n=18$ and $a=1$ by Lemma \ref{Counting}. This graph is indeed the complete bipartite graph $K_{10,10}$ by Remark \ref{Complete multipartite graph} which is a strongly regular Neumaier Cayley graph. If $c=3$, then there are two possible cases by Lemma \ref{Counting}; firstly $n=27$ and $a=1$ which is impossible by Equation \ref{equation 1}, and secondly $n=15$ and $a=2$. If $n=15$, then it has parameters $(15,10,5;1,3)$ by Equation \ref{equation 1} and it is indeed the complete multipartite graph $K_{5,5,5}$ which is a strongly regular Neumaier Cayley graph (see Example \ref{Example}).
If $c=4$, then there are two possible cases by Lemma \ref{Counting}; firstly $n=32$ and $a=1$, secondly $n=18$ and $a=2$ which is impossible by Equation \ref{equation 2}. If $n=32$, then it has parameters $(32,10,2;1,4)$ but there is neither strictly Neumaier graph with such parameters (cf. \cite[Table~1]{ACDKZ}) nor strongly regular graph (cf. \cite[Chap.~12]{BM}).
If $c=5$, then there are three possible cases by Lemma \ref{Counting}; firstly $n=35$ and $a=1$, secondly $n=20$ and $a=2$ which is impossible by Equation \ref{equation 2}, and thirdly $n=15$ and $a=3$ which is impossible by Theorem \ref{The number of edges}. If $n=35$, then it has parameters $(35,10,3;1,5)$.
\begin{lem}
There is no Neumaier Cayley graph with parameters $(35,10,3;1,5)$.
\end{lem}
\begin{proof}
Let $Cay(G,S)$ be a Neumaier Cayley graph with parameters $(35,10,3;1,5)$ and $C$ the regular clique containing the identity element $e$. Then $G \cong \mathbb{Z}_{35}$ because the cyclic group of order $35$ is the only group of order $35$. We first prove that the regular clique $C$ is the subgroup of the group $G$. Let $G=\seq{a}$ and $C=\{e,a^{i},a^{-i},a^{j},a^{-j}\}$. Then
\begin{equation*}
\{a^{i},a^{-i},a^{2i},a^{-2i},a^{i-j},a^{-i+j},a^{i+j},a^{-i-j},a^{j},a^{-j},a^{2j},a^{-2j}\} \subseteq S,
\end{equation*}
Moreover, $a^{i-j}$ is adjacent to $e$, $a^{i}$, and $a^{-j}$ and therefore $a^{i-j} \in C$ since the nexus of the regular clique $C$ is $1$. This implies that $a^{i-j}=a^{-i}$ or $a^{i-j}=a^{j}$. With out loss of generality, we can suppose that $a^{i-j}=a^{-i}$ and therefore $C=\{e,a^{i},a^{2i},a^{-i},a^{-2i}\}$. Furthermore, $a^{3i}$ is adjacent to $e$, $a^{i}$, and $a^{2i}$ and therefore $a^{3i} \in C$ because the nexus of the regular clique $C$ is $1$. This implies that $a^{3i}=a^{-2i}$ and $i=7$. It follows that $C=\{e,a^{7},a^{14},a^{21},a^{28}\}$ which is the unique subgroup of order $5$ of the group $G$. It is easy to see that each element of $P=\{C,Ca,Ca^{2},Ca^{3},Ca^{4},Ca^{5},Ca^{6}\}$ is a regular clique with nexus $1$ for the Cayley graph $Cay(G,S)$ and the set $P$ is a partition for the group $G$. This implies that
\begin{equation*}
S=\{a^{7},a^{14},a^{21},a^{28},a^{i},a^{-i},a^{j},a^{-j},a^{k},a^{-k}\},
\end{equation*}
where $a^{i} \in Ca$, $a^{j} \in Ca^{2}$, and $a^{k} \in Ca^{3}$. If $a^{i}$ is adjacent to $a^{-i}$, then $a^{2i} \in S$. This implies that, with out loss of generality, $j=2i$. Moreover, $a^{i}$ is adjacent to $a^{k}$ since $\lambda=3$. It follows that $i-k=-2i$ and therefore $k=3i$, a contradiction with $\lambda=3$. This implies that $a^{i}$ is not adjacent to $a^{-i}$. If $a^{i}$ is adjacent to $a^{j}$, then $i-j=k$ and
\begin{equation*}
S=\{a^{7},a^{14},a^{21},a^{28},a^{i},a^{-i},a^{j},a^{-j},a^{i-j},a^{j-i}\},
\end{equation*}
a contradiction with $\lambda=3$ because the common neighbors of the identity element $e$ and $a^{i}$ can be $a^{j}$ and $a^{i-j}$ and this completes the proof.
\end{proof}
If $c=6$, then there are four possible cases by Lemma \ref{Counting}; firstly $n=36$ and $a=1$, secondly $n=21$ and $a=2$, thirdly $n=16$ and $a=3$ , and fourthly $n=12$ and $a=5$. If $n=36$, then it is easy to see that it is the lattice graph $L_{2}(6)$ which is a strongly regular Neumaier Cayley graph (see Example \ref{Example}). If $n=21$, then it is the triangular graph $T(7)$ which is a strongly regular Neumaier graph with parameters $(21,10,5,4;6,2)$. This graph is indeed a Cayley graph $Cay(G,S)$ over the group $G=\seq{a,b|a^{3}=b^{7}=1,a^{-1}ba=b^{2}}$ with the connection set
\begin{equation*}
\{a,a^{-1},b^{2},b^{-2},ab^{2},b^{-2}a^{-1},ab^{4},b^{-4}a^{-1},a^{2}b,b^{-1}a^{-2}\}.
\end{equation*}
If $n=16$, then it is indeed the complement of Clebsch graph which is a strongly regular graph. We checked it with  \verb"GAP" \cite{GAP} and it is not a Neumaier graph. If $n=12$, then it is the complete multipartite graph $K_{2,2,2,2,2,2}$ which is a strongly regular Neumaier Cayley graph (see Example \ref{Example}).
If $c=7$, then there are three possible cases by Lemma \ref{Counting}; firstly $n=35$ and $a=1$ which is impossible by Remark \ref{nexus one}, secondly $n=21$ and $a=2$, and thirdly $n=14$ and $a=4$ which are impossible by Equation \ref{equation 2}.
If $c=8$, then there is only one possible case by Lemma \ref{Counting} that is $n=32$ and $a=1$ which is impossible by Remark \ref{nexus one}.
Finally, if $c=9$, then there are two possible cases by Lemma \ref{Counting}; firstly $n=27$ and $a=1$ which is impossible by Remark \ref{nexus one}, and secondly $n=18$ and $a=2$ which is impossible by Equation \ref{equation 2}.
\section{Neumaier graphs with diameter Three}
There are a few examples of Neumaier graphs with diameter three. Two of them appeared in Evans' thesis; these are Neumaier graphs with parameters $(24,8,2;1,4)$ (cf. \cite[\S~4.4.2]{Evans}). Recall that a Neumaier Cayley graph has diameter two and therefore there is no Neumaier Cayley graph with diameter three (see Observation \ref{Diameter 2}).
\section*{Acknowledgements}
The author is grateful to the Research Council of Shahid Chamran University of Ahvaz for financial support (SCU.MM1401.29248).


\begin{thebibliography}{99}
\bibitem{ADJ}
A. Abdollahi, E.R. van Dam and M. Jazaeri, Distance-regular Cayley graphs with least eigenvalue $-2$, Des. Codes Cryptogr. {\bf 84} (2017) 73--85.

\bibitem{AJ}
A. Abdollahi and M. Jazaeri, On groups admitting no integral Cayley graphs besides complete multipartite graphs, Appl. Anal. Discrete Math. {\bf 7} (2013) 119--128.

\bibitem{ACDKZ} A. Abiad, W. Castryck, M. DeBoeck, J. H. Koolen, and S. Zeijlemaker, An infinite class of Neumaier graphs and non-existence results, J. Combin. Theory Ser. A, {\bf 193} (2023) 105684.

\bibitem{ADDK} A. Abiad, B. De Bruyn, J. D'haeseleer, and J. H. Koolen, Neumaier graphs with few eigenvalues, Des. Codes Cryptogr., {\bf 90} (2022) 2003--2019.

\bibitem{Evans} R.J. Evans, On regular induced subgraphs of edge-regular graphs, Queen Mary University of London, Ph.D thesis (2020).

\bibitem{Eagt} R.J. Evans, AGT, Algebraic Graph Theory, Version 0.3.1 (2022), \url{https://gap-packages.github.io/agt/}.

\bibitem{EGP} R.J. Evans, S. Goryainov, D. Panasenko, The smallest strictly Neumaier graph and its generalisations, Electron. J. Comb., {\bf 26(2)} (2019).

\bibitem{BCN} A.E. Brouwer, A.M. Cohen, and A. Neumaier, Distance-regular graphs, Springer, New York, 1989.

\bibitem{BH} A.E. Brouwer and W.H. Haemers, Spectra of graphs, Springer, New York, 2012.

\bibitem{BM} A.E. Brouwer, H. Van Maldeghem, Strongly Regular Graphs, Encyclopedia of Mathematics and Its Applications, Cambridge University Press, 2022.

\bibitem{Dam} E. R. van Dam, Three-class association schemes, J. Algebraic Combin., {\bf 10} (1999) 69--107.

\bibitem{DJ} E. R. van Dam and M. Jazaeri, Distance-regular Cayley graphs with small valency, Ars Math. Contemp., {\bf 17} (2019) 203--222.

\bibitem{GK} G. R. W. Greaves and J. H. Koolen, Edge-regular graphs with regular cliques, Europ. J. Combin., {\bf 71} (2018) 194--201.

\bibitem{MA} S. L. MA, A survey of partial difference sets, Des. Codes Cryptogr. {\bf 4} (1994) 221–-261.

\bibitem{Neumaier} A. Neumaier, Regular Cliques in graphs and Special $1\frac{1}{2}$-designs, Finite Geometries
and Designs, London Mathematical Society Lecture Note Series, (1981) 245--259.

\bibitem{Soicher} L.H. Soicher, On cliques in edge-regular graphs, J. Algebra, {\bf 421} (2015) 260--267.

\bibitem{Sgrape} L. H. Soicher, GRAPE, GRaph Algorithms using PErmutation groups, Version 4.9.0 (2022),
\url{https://gap-packages.github.io/grape}.

\bibitem{GAP} The GAP Group, GAP -- Groups, Algorithms, and Programming, Version 4.12.1 (2023),
\url{http://www.gap-system.org}.
\end{thebibliography}
\end{document}